\documentclass[journal]{IEEEtran}
\usepackage{Lilypack}

\usepackage{amsmath}
\usepackage{amsthm}
\usepackage{amssymb}
\newtheorem{proposition}{Proposition}
\newtheorem{assumptionletter}{Assumption}

\usepackage{balance}

\usepackage{pgfplots}
\usepackage{enumerate}
\newcommand{\Prb}{\mathrm{P}}

\newcommand{\thetaH}{{\theta}_H}
\newcommand{\thetaL}{{\theta}_L}
\newcommand{\hU}{\hat{U}}

\newcommand{\q}{x}

\newcommand{\hx}{\hat{x}}
\newcommand{\htt}{\hat{t}}

\begin{document}
\title{Effects of Risk on Privacy Contracts for Demand-Side Management}
\newif\ifarxiv
\arxivtrue

\ifarxiv
\author{
	Lillian J. Ratliff,  Carlos Barreto,      Roy~Dong,
  Henrik~Ohlsson,
        Alvaro~A.~C\'ardenas,
        and~S.~Shankar~Sastry,% <-this % stops a space
\thanks{L.~J.~Ratliff, R.~Dong, H.~Ohlsson, and S.~S.~Sastry are with the 
Department of Electrical Engineering and Computer Sciences, 
University of California, Berkeley, CA, USA.
e-mail: \texttt{\{roydong,ratliffl,ohlsson,sastry\}@eecs.berkeley.edu}.}% <-this % stops a space
\thanks{C. Barreto, and A.~A.~C\'ardenas are with the Department of Computer Science, University of Texas at Dallas, Richardson, TX, USA.
email: \texttt{\{carlos.barretosuarez,alvaro.cardenas\}@utdallas.edu}.}
\thanks{The work presented is supported by the NSF
Graduate Research Fellowship under grant DGE 1106400, NSF
CPS:Large:ActionWebs award number 0931843, TRUST (Team for Research in
Ubiquitous Secure Technology) which receives support from NSF (award
number CCF-0424422), and FORCES (Foundations Of Resilient
CybEr-physical Systems) CNS-1239166, the European Research Council
   under the advanced grant LEARN, contract 267381, a postdoctoral grant from the Sweden-America
   Foundation, donated by ASEA's Fellowship Fund, and  by a postdoctoral
   grant from the Swedish Research Council.
}}
\else
\author{
	Lillian J. Ratliff,~\IEEEmembership{Student~Member,~IEEE,}
  Carlos Barreto,~\IEEEmembership{Student~Member,~IEEE,}
      Roy~Dong,~\IEEEmembership{Student~Member,~IEEE,}
         Henrik~Ohlsson,~\IEEEmembership{Member,~IEEE,}
        Alvaro~A.~C\'ardenas,~\IEEEmembership{Member,~IEEE,}
        and~S.~Shankar~Sastry,~\IEEEmembership{Fellow,~IEEE}% <-this % stops a space
\thanks{L.~J.~Ratliff, R.~Dong, H.~Ohlsson, and S.~S.~Sastry are with the 
Department of Electrical Engineering and Computer Sciences, 
University of California, Berkeley, CA, USA.
e-mail: \texttt{\{roydong,ratliffl,ohlsson,sastry\}@eecs.berkeley.edu}.}% <-this % stops a space
\thanks{C. Barreto, and A.~A.~C\'ardenas are with the Department of Computer Science, University of Texas at Dallas, Richardson, TX, USA.
email: \texttt{\{carlos.barretosuarez,alvaro.cardenas\}@utdallas.edu}.}
\thanks{The work presented is supported by the NSF
Graduate Research Fellowship under grant DGE 1106400, NSF
CPS:Large:ActionWebs award number 0931843, TRUST (Team for Research in
Ubiquitous Secure Technology) which receives support from NSF (award
number CCF-0424422), and FORCES (Foundations Of Resilient
CybEr-physical Systems) CNS-1239166, the European Research Council
   under the advanced grant LEARN, contract 267381, a postdoctoral grant from the Sweden-America
   Foundation, donated by ASEA's Fellowship Fund, and  by a postdoctoral
   grant from the Swedish Research Council.
}}
\fi

%\markboth{Journal of \LaTeX\ Class Files,~Vol.~11, No.~4, December~2012}%
%{Shell \MakeLowercase{\textit{et al.}}: Bare Demo of IEEEtran.cls for Journals}

\maketitle

\begin{abstract}
As smart meters continue to be deployed around the world collecting unprecedented levels of fine-grained data about consumers, we need to find mechanisms that are fair to both, (1) the electric utility who needs the data to improve their operations, and (2) the consumer who has a valuation of privacy but at the same time benefits from sharing consumption data.  In this paper we address this problem by proposing privacy contracts between electric utilities and consumers with the goal of maximizing the social welfare of both. Our mathematical model designs an optimization problem between a population of users that have different valuations on privacy and the costs of operation by the utility. We then show how contracts can change depending on the probability of a privacy breach.  This line of research can help inform not only current but also future smart meter collection practices.
\end{abstract}

%\begin{IEEEkeywords}
%privacy, game theory, control system human factors
%\end{IEEEkeywords}

% \fi
\IEEEpeerreviewmaketitle

\section{Introduction}
\label{sec:introduction}
\IEEEPARstart{I}{ncreasingly} advanced metering infrastructure (AMI) is
replacing older technology in the electricity grid. Smart meters measure detailed
information about consumer electricity usage every half-hour, quarter-hour, or in some cases, every five minutes. This high-granularity data is needed to support energy efficiency efforts as well as demand-side management.   
However, improper handling of this 
information could also lead to unprecedented invasions of consumer
privacy~\cite{quinn:2009aa, salehie:2012aa, wicker:2011aa, hart:1989aa}. 

It has been shown that energy consumption data reveals a considerable amount
about consumer activities. Furthermore, energy consumption data in combination
with readily available sources of information can be used to discover even more
about the consumer. Authors in~\cite{lisovich:2010aa} argue and experimentally
validate that a privacy
breach can be broadly implemented in two steps. First, energy usage data in combination
with other sensors in the home --- e.g. water and gas usage --- can be used to
track a person's location, their appliance usage, and match
individuals to observed events. In the second step, this learned
information can be combined with demographic data --- e.g. number, age, sex of
individuals in the residence --- to identify activities, behaviors, etc.

Given that smart grid operations inherently have privacy and security risks~\cite{salehie:2012aa},
%a task such as direct load control has both privacy and security risks, 
it would benefit the utility company, to know the answer to the following questions: How do consumers in the population value privacy? How can we quantify privacy? 
%%What is the probability that an adversary will fail to infer private information from the consumption data? 
How do privacy-aware policies impact smart grid operations? 
%%In this paper we address these questions as well as expose new directions for future research on privacy and customer segmentation in the smart grid. 
%
%%The goal being to protect consumers against attacks from adversaries listening in on the line to consumption data and learning things about when agents are home, etc. 
%
%Using our results on the fundamental limits of non-intrusive load
%monitoring~\cite{dong:2013aa,dong:2014ab, ratliff:2014aa}, we have an upper
%bound on the probability of successful implementation of a privacy breach by an
%adversary. This bound is independent of the algorithm used by the adversary for
%distringuishing between two hypotheses. 
There have been a nuber of works making efforts to address these
questions~\cite{dong:2013aa, dong:2014ab, sankar:2013aa, rajagopalan:2011aa}. In
particular, it has been shown that there is a fundamental utility-privacy
tradeoff in data collection policies in smart grid operations. 
%Further, a number
%of privacy metrics have been proposed. 

%Given that a utility company desires to --- e.g. demand
%response program or direct load control (DLC) --- 
%Given a privacy metric and a smart grid operation, e.g. demand response or
%direct load contorl, the utility company 
%Given the utility company wants
A utility company that desires to conduct some smart grid operation, such as
demand response or direct load control program, requires
%. In order to do so, they require
high-fidelity data. 
%The consumer may be unwilling to exchange their privacy for
%the sake of efficient operation of the utility company's program. We propose
%that the utility company can design a screening mechanism to assess %that lets them assess
%how consumers value privacy. 
They can observe consumers' electricity usage
but do not know their privacy preferences. We propose
that the utility company can design a screening mechanism --- in particular, a menu of contracts --- to assess %that lets them assess
how consumers value privacy.

In general, contracts are essential for realizing the benefits of economic exchanges including those made in smart grid operations.
Contract theory has been used in energy grid applications including procurement of electric
power from a strategic seller~\cite{tavafoghi:2014aa} and demand response
programs~\cite{fahrioglu:2000aa, fahrioglu:1999aa} among others~\cite{gedra:1994aa}. We take a novel view point by considering privacy to be the good on which we design contracts.

We design contracts given the utility company has an arbitrary smart grid operation they want to implement yet the consumer's preferences are unknown. 
%with are individually rational and incentive compatible employ the theory of contracts for 
 In particular, based on their valuation of privacy as a good, consumers can select the quality of the service contract with the utility company. Essentially, electricity service is offered as a product line differentiated according to privacy where consumers can select the level of privacy that fits their needs and wallet. The optimal contracts are incentive compatible and individually rational.
 
Further, we assess loss risks due to privacy breaches given the optimally designed contracts. We design new contracts when these risks are explicitly considered by the utility company. We provide a characterization of the contracts designed with and without loss risks. We show that there are inefficiencies
when we consider losses incurred due to privacy breaches and thus, in some
cases, the utility company has an incentive to offer compensation to the consumer in the event of a privacy breach, invest in security measures, and purchase insurance. 
%for use by the
%utility company for smart grid operations.
%programs like demand response, direct load control, etc. In addition, third-party insurance companies can design insurance contracts. Insurance allows the consumer to protect himself in the event of a privacy breach, i.e. she will be compensated for any experienced loss. 

The purpose of this paper is to provide qualitative assessment of privacy
contracts for demand-side management through the use
of simple examples which have all the interesting properties of the
larger problem such as asymmetric information. 
%In Section~\ref{sec:privacy}, we review literature in contract theory applied to the smart grid.
In Section~\ref{sec:contracts}, we study a two-type model; there are two types of
consumers --- ones that have a high valuation of privacy and ones that have a
low valuation of privacy --- and characterize the solution for the contract design
problem. 
%It is known in the theory of contracts that the
%fundamental characteristics of the two-type model are the same for the case when there is a
%continuum of types (see, e.g.,~\cite{bolton:2005aa}). As we will show in the
%case where privacy is the good, the high-type recieves an efficient
%allocation and positive information rent whereas the low-type gets zero
%surplus. %, and all types except the
%low-type receive positive information rents. 
In Section~\ref{sec:risk}, we characterize the
contract solution when the consumer is \emph{risk-adverse} and the risk of a privacy breach is
explicitly modeled in the contract design.  In
Section~\ref{sec:dlc}, we present an example where the utility company is
interested in implementing a direct load control scheme and designs contracts
with and without risk. Finally, in Section~\ref{sec:discussion}, we provide discussion and future research directions.% questions.
\section{Privacy Contracts}
\label{sec:contracts}
In this section, we discuss the design of privacy-based contracts that are offered
to consumers, who possess private information, i.e., the utility does not know the characteristics of each user.  
We consider a model in which there are only two classes of users and we utilize standard
results from the theory of screening (see, e.g.,~\cite{weber:2011aa,
bolton:2005aa}) to develop
a framework for designing privacy contracts. In general, the fundamental
characteristics of the two-type problem extend to the any number of types
including a continuum of types. 
%We remark that as a result of the screening
%process the utility company will know how each consumer values privacy and can
%leverage that in the design of incentives aimed at inducing the consumer to
%select a privacy setting more desirable from the perspective of the utility
%company. In addition, the screening process can be thought of as customer
%segmentation since it will extract each consumer's type which can be used for
%segmentation.% after which the consumers can be grouped according to their type.

We model privacy-settings on smart meters (e.g., sampling rate) as a good. The \emph{quality} of the
good is either a high-privacy setting $\q_H$ or a low-privacy setting $\q_L$, which can be interpreted for example as low and high sampling rates, respectively.
%The consumer can choose either a high privacy setting or a low privacy setting,
%i.e. 
Each consumer selects $\q\in\mc{X}=\{\q_H, \q_L\}\subset \mb{R}$ where $-\infty<\q_L<\q_H<\infty$.
  On the other hand, the consumer's valuation of privacy is 
  characterized by the parameter $\theta$, commonly called the  \emph{type} of a user. The type of an agent is seen as a piece of private information that determines the willingness of a user to pay for a good: in our privacy setting, the type characterizes the electricity consumption privacy needs of the user.
 In our model, we assume $\theta$ takes only two values, i.e.,   
  $\theta\in\{\thetaL, \thetaH\} \subset \mb{R}$, where 
  %$\theta$ represents how   much the consumer values high-privacy over low-privacy. We assume 
  $\thetaL<\thetaH$. The type $\theta$ is distinct from the private information that is subject to a privacy breach.
 % by this we mean that how much the consumer values privacy is not also private  information. 
 %We note here that these types implicitly make use of the metrics
 %for privacy presented in Section \ref{sec:privacy}.

The consumers type $\theta$ is related to his willingness to pay in the following way: if the utility company announces a price $t$ for choosing $\q$, the type-dependent consumer's utility is equal to zero if he does not select a privacy setting $\q$, and
\begin{equation}
  \hU(\q,\theta)-t\geq 0
  \label{eq:ir}
\end{equation}
if he does select a privacy setting. The case in which the consumer does not select a privacy setting is considered the \emph{opt-out} case in which consumer exercises his right to not participate. 
The inequality in \eqref{eq:ir} is often called the \emph{individual
rationality} constraint. 

We have the following assumptions on the consumer's
utility function which represents its preferences:
\begin{assumptionletter}
  \label{ass:U_concave}
The utility function $\hat{U}:\mathbb{R} \times \Theta \rightarrow \mathbb{R}$ is strictly increasing in $(x, \theta) \in \mathbb{R} \times \Theta$, concave in $x \in \mathbb{R}$, and differentiable with respect to $x$.
\end{assumptionletter}
  \begin{assumptionletter}
    \label{ass:marginal}
  The marginal gain from raising the value of the privacy
setting $\q$ is greater for type $\thetaH$, i.e. $\hU(\q, \thetaH)-\hU(\q, \thetaL)$ is increasing in $\q$.
  \end{assumptionletter}
 % The second assumption is often referred to as the \emph{super-modularity}
 % assumption. 

Since we have only two types, the contracts offered will be indexed by the privacy settings $\q_L$ and $\q_H$.
Further, as we mentioned before, the consumer can opt-out by not selecting a privacy option at all. Hence, we need to constrain the mechanism design problem by enforcing the inequality given in Equation~\eqref{eq:ir} for each value of $\theta\in \{\thetaL, \thetaH\}$. In addition, we need to enforce \emph{incentive-compatibility} constraints  
\begin{equation}
  \hU(\q_H, \thetaH)-t_H\geq \hU(\q_L, \thetaH)-t_L
  \label{eq:ic1}
\end{equation}
and 
\begin{equation}
  \hU(\q_L, \thetaL)-t_L\geq \hU(\q_H, \thetaL)-t_H
  \label{eq:ic2}
\end{equation}
where the first inequality says that given the price $t_H$ a consumer of type $\thetaH$ should prefer the high-privacy setting $\q_H$ and the second inequality says that given the price $t_L$ a consumer of type $\thetaL$ should prefer the low-privacy setting $\q_L$.

The utility company has unit utility
\begin{equation}
  v(\q, t)=-g(\q)+t
  \label{eq:uc_utility}
\end{equation}
where $g:\mc{X}\rar\mb{R}$ is the unit cost experienced  by the utility company with a privacy setting $\q$, which satisfies the following assumption.
\begin{assumptionletter}\label{as:g_convex}
The cost function $g:\mathbb{R} \rightarrow \mathbb{R}$ is a strictly increasing, convex, and differentiable function.
\end{assumptionletter}
This assumption is reasonable since
%as we have
%mentioned in Section \ref{sec:privacy}, 
a low-privacy setting $\q_L$ provides the utility company with the high-granularity data it needs to efficiently operate and maintain the smart grid~\cite{dong:2014ab, sankar:2013aa, rajagopalan:2011aa}. 
%For example, recall Section~\ref{subsec:dlc} in which we show that the performance of DLC degrades with decreases in sampling rate.

%We are currently exploring methods to design a metric for how access to high-granularity data affects certain tasks like \emph{direct load control}, \emph{demand response programs}, or \emph{incentive/rebate programs}. In future work, this metric will replace $g$; however, for this paper, we leave $g$ to be an abstract object. 
%Suppose the utility company knew the type of the consumer it was facing. Then
%the optimization problem that it solves is
%\begin{equation}
%  \max_{(x,t)} v(x,t)
%  \label{<++>}
%\end{equation}<++>

Let the expected profit of the utility company be given by
\begin{equation}
  \Pi(t_L, \q_L, t_H, \q_H)=(1-p)v(\q_L,t_L)+pv(\q_H, t_H)%(1-p)(g(x_L)-t_L)+p(g(x_H)-t_H)
  \label{eq:profit1}
\end{equation}
where $p=\Prb(\theta=\thetaH)=1-\Prb(\theta=\thetaL)\in (0, 1)$ and $\Prb(\cdot)$ denotes probability.
The screening problem is to design the contracts, i.e. $\{(t_L, \q_L), (t_H,
\q_H)\}$ where $t_L, t_H\in \mb{R}$, so that the utility companies expected
profit is maximized.
In particular, to find the optimal pair of contracts, we solve the following optimization problem:
\begin{align} 
  \max_{\{(t_L, \q_L), (t_H, \q_H)\}}& \Pi(t_L, \q_L, t_H, \q_H)\tag{P-1} \label{eq:P1}
\\
\text{s.t.}\qquad&\ \hU(\q_H, \thetaH)-t_H\geq \hU(\q_L, \thetaH)-t_L\tag{IC-1}\label{eq:IC1}\\
&\ \hU(\q_L, \thetaL)-t_L\geq \hU(\q_H, \thetaL)-t_H\tag{IC-2}\label{eq:IC2}\\
&\ \hU(\q_L,\thetaL)-t_L\geq 0\tag{IR-1}\label{eq:IR1}\\
&\ \hU(\q_H,\thetaH)-t_H\geq 0\tag{IR-2}\label{eq:IR2}\\
&\ \q_L\leq \q_H\notag
 \end{align}

 Depending on the form of $\hU(\q, \theta)$ and $g(\q)$ problem~\eqref{eq:P1} can
 be difficult to solve. Hence, we reduce the problem using characteristics of
 the functions and constraints. We remark that this process of reducing the constraint set in contract design with a finite number of types is well known (see, e.g., ~\cite{basov:2005aa,bolton:2005aa,weber:2011aa}) and sometimes referred to as the constraint reduction theorem.
 %So, we examine the constraints and try to eliminate as many as we can. 
% First, we show that \eqref{eq:IR2} is redundant. Indeed,
% \begin{equation}
%    U(x_H, \thetaH)-t_H\geq U(x_L, \thetaH)-t_L\geq  U(x_L, \thetaL)-t_L\geq0
%   \label{eq:b1}
% \end{equation}
% where the first inequality is \eqref{eq:IC1}, the second holds since $U(x, \theta)$ is increasing in $\theta$ by assumption, and the last is \eqref{eq:IR1}. Now, \eqref{eq:IR1} must be active since if not, $t_H$ and $t_L$ could be increased by $\vep>0$ without violating any of the constraints thereby increasing the utility company's payoff.  
 
First, we show that \eqref{eq:IR1} is active. Indeed, suppose not. Then, $\hU(\q_L, \thetaL)-t_L>0$ so that, from the first incentive compatibility constraint \eqref{eq:IC1}, we have
\begin{equation}
  \hU(\q_H, \thetaH)-t_H\geq \hU(\q_L, \thetaH)-t_L\geq  \hU(\q_L, \thetaL)-t_L>0
  \label{eq:binding}
\end{equation}
where the second to last inequality holds since $\hU(\q, \theta)$ is increasing in $\theta$ by assumption. 
As a consequence, the utility company can increase the price for both types since neither incentive compatibility constraint would be active. This would lead to an increase in the utility company's pay-off and thus, we have a contradiction. 

Now, since $\hU(\q_L, \thetaL)=t_L$, the last inequality in \eqref{eq:binding} is equal to zero. This implies that \eqref{eq:IR2} is redundant.
% Now, $U(x_L, \thetaL)=t_L$, so that
%\begin{equation}
%   U(x_H, \thetaH)-t_H\geq U(x_L, \thetaH)-t_L\geq  U(x_L, \thetaL)-t_L=0
%  \label{eq:b4}
%\end{equation}
%implies that constraint \eqref{eq:IR2} can be removed. 
Further, this argument implies that the constraint \eqref{eq:IC1} is active. 
Indeed, again suppose not. Then, 
\begin{equation}
   \hU(\q_H, \thetaH)-t_H> \hU(\q_L, \thetaH)-t_L\geq  \hU(\q_L, \thetaL)-t_L=0
  \label{eq:binding2}
\end{equation}
so that it would be possible for the utility company to decrease the incentive $t_H$ without violating \eqref{eq:IR2}. 

%Now, let us assume that the marginal gain from raising the value of the privacy
%setting $\q$ is greater for type $\thetaH$, i.e. $\hU(\q, \thetaH)-\hU(\q, \thetaL)$ is increasing in $\q$.  
By Assumption~\ref{ass:marginal} and the fact that
 \eqref{eq:IC1} is active, we have
\begin{equation}
  t_H-t_L=\hU(\q_H, \thetaH)-\hU(\q_L, \thetaH)\geq \hU(\q_H, \thetaL)-\hU(\q_L, \thetaL).
  \label{eq:binding3}
\end{equation}
This inequality implies that we can ignore \eqref{eq:IC2}. Further, since $\hU$ is increasing in $(\q, \theta)$ and by assumption $\thetaH>\thetaL$, we have that $\q_L< \q_H$. 

Thus, we have reduced the constraint set:
\begin{align}
  t_H-t_L&=\hU(\q_H, \thetaH)-\hU(\q_L, \thetaH),\label{eq:newconst1}\\
 t_L&= \hU(\q_L, \thetaL).
 \label{eq:newconst2}
\end{align}
%Thus, the optimization problem becomes 
%\begin{align}
%  \label{eq:newopt}
%  \max_{(\q_L, \q_H)}&\bigg\{(1-p)(U(\q_L, \thetaL)-g(\q_L))\tag{P-2}\\
%  &\quad +p(U(\q_H, \thetaH)-g(\q_H)-U(\q_L, \thetaH)+U(\q_L, \thetaL))\bigg\}\notag
%\end{align}
The optimization problem (P-1) reduces to two independent optimization problems:
\begin{flalign}
  \label{eq:P3a}
  \hspace{20pt} \max_{\q_L}\{\hU(\q_L, \thetaL)-(1-p)g(\q_L)-p\hU(\q_L, \thetaH)\} && \tag{P-2a}
\end{flalign}
and
\begin{flalign}
  \label{eq:P3b}
  \hspace{20pt} \max_{\q_H}\{\hU(\q_H, \thetaH)-g(\q_H)\}. && \tag{P-2b}
\end{flalign}
We will denote the solution to the two optimization problems above by
$(\hat{x}^\ast_i,\hat{t}^\ast_i)$ and for reasons that will become apparent in
the next paragraph we call it the \emph{second-best} solution.

We now claim that the optimal contract satisfies $\hx_L^\ast<\hx_H^\ast$. Indeed,
consider the case when the utility company knows the type of the consumer
that it faces, i.e. the solution under full information which we refer to as the
\emph{first-best} solution. We denote the first-best solution by $(\hx_i^\dagger,
\htt_i^\dagger)$
for $i\in\{L,H\}$ where the pair solves
\begin{equation}
  \max_{\hx_i} \hU(\hx_i, \theta_i)-g(\hx_i)
  \label{eq:fbsol}
\end{equation}
and $\htt_i^\dagger=\hU(\hx_i^\dagger, \theta_i)$. Note that throughout the rest
of the paper we will use the notation $(\cdot)^\dagger$ to denote the first-best
solution and $(\cdot)^\ast$ to denote the second-best solution.

Notice that~\eqref{eq:P3b} is exactly the optimization problem the utility
  company would solve to determine the first-best solution for the high-type.
  Hence, even when there is hidden information, the high-type will always be
  offered the first-best quality and first-best price, i.e.
  $(\hat{x}_H^\ast,\hat{t}_H^\ast)=(\hat{x}_H^\dagger, \hat{t}_H^\dagger)$. This
  is to say that the high-type receives an efficient allocation.
  %of the good
  %which is privacy.
  
  Assumption~\ref{ass:marginal} implies that the first-best solution $\hx_i^\dagger(\theta)$ is increasing
  in $\theta$. Further, $\hU(\hx_L, \thetaH)-\hU(\hx_L, \thetaL)$ is increasing
  in $\hx_L$ and non-negative so that $\hx^\dagger_L\geq \hx_L^\ast$. Thus
   $\hx_L^\ast\leq \hx_L^\dagger<\hx_H^\dagger=\hx_H^\ast$.

  This result also shows that when there is asymmetric information, the
  low-type gets zero surplus and a quality level that is inefficient when compared
  to the first-best solution. On the other hand, as we have pointed out, the
  high-type is offered the first-best quality and has more surplus. Further, the
  high-type gets positive \emph{information rent}:
  \begin{equation}
    \htt^\ast_H=\htt_H^\dagger-\underbrace{\left( \hU(\hx_L^\ast,
    \thetaH)-\hU(\hx_L^\ast, \thetaL) \right)}_{\text{\emph{information rent}}}.
    \label{eq:inforent}
  \end{equation}
  We will see the effects of this in the example presented in
  Section~\ref{sec:dlc}.

\section{Effects of Risk on Privacy Contracts}
\label{sec:risk}
%\subsection{Effects of Risk}

%Previous work \cite{cdc} shows the design of privacy contracts for demand response programs. 
In this section, we are interested in analyzing the effect of loss risk (due to privacy breaches) in contracts. 
%Specifically, we want to show how privacy risks affect the contract design.

Let us consider that users might suffer privacy breaches of cost $\ell(\theta)$, with probability $1-\eta(x)$. Then, their expected profit can be expressed as:
\begin{equation}\label{eq:U_risk}
U(x,\theta) = \hat{U}(x, \theta) - (1-\eta(x)) \ell(\theta).
\end{equation}
The characteristics of the privacy breach are summarized in the following assumption.
\begin{assumptionletter}\label{as:eta}
$\eta:\mathbb{R} \rightarrow [0,1]$ (probability of avoiding a privacy breach)
is strictly increasing with respect to the privacy $x$. 
The perceived loss due to a privacy breach $\ell:\Theta \rightarrow \mathbb{R}_{\geq 0}$ is increasing with respect to the type of each user. 
\end{assumptionletter}

Roughly speaking, the higher the privacy setting, the less likely a privacy
breach will occur.
%a user with high privacy avoids revelation of senstive information that might be used to perform a privacy breach. 
Furthermore, a user with high privacy needs might experience smaller costs compared to a user with low privacy settings.

%Let us first study what happens when the utility company issues the optimal
%contracts when risk is not considered in the design of the contracts yet
%the consumers have privacy risks. 

The individual rationality constraints for the case where consumers have privacy
risks can be expressed as
\begin{equation}
  \hU(x, \theta)-t\geq (1-\eta(x)) \ell(\theta).
  \label{eq:IRwrisk}
\end{equation}
Recall that the optimal contract without risk for the low-type,
$(\hat{x}_L^\ast, \hat{t}_L^\ast)$, satisfies \eqref{eq:IR1} with strict
equality, i.e. $\hU(\hx_L^\ast, \thetaL)=\hat{t}_L^\ast$. Thus, the optimal
contract of the previous section violates \eqref{eq:IRwrisk} unless either
$\ell(\thetaL)=0$ or $\eta(\hx_L^\ast)=1$. Consequently, consumers with low
privacy preferences $\thetaL$ might do better by opting out. % rejecting all offered con
%Now, recall that the optimal contract (that does not consider risk) satisfies the \emph{individual rationality} condition in Eq. (\ref{eq:IR1}) with strict equality.
%Thus, if users have privacy risks, then Eq. (\ref{eq:IR1})  can be expressed as
%\begin{equation}
%\hat{U}(x_L, \thetaL) - t_L \geq (1-\eta(x_L)) \ell(\thetaL). 
%\end{equation}
%However, this condition is not satisfied by the contract (see Eq. (\ref{eq:IR1})), and consequently, users with low privacy preferences $\thetaL$ might do better by rejecting all offered contracts, i.e., opting-out. 

On the other hand, the \emph{incentive compatibility} constraint \eqref{eq:IC2} can
be expressed as 
%for the case with
%risk can be expressed as
\begin{equation}\label{eq:IC2_risk}
\hat{U}(x_L, \thetaL) - t_L \geq \hat{U}(x_H, \thetaL) - t_H  + (\eta(x_H) - \eta(x_L)) \ell(\thetaL)
\end{equation}
when we consider privacy risks.
Note that since $\hx_L^\ast<\hx_H^\ast$, $\eta(\hx_H^\ast) - \eta(\hx_L^\ast) > 0$. 
Hence, the inequality in \eqref{eq:IC2_risk} might not be satisfied by the
optimal contract that does not consider risk. Consequently, consumers with low
preferences $\thetaL$ might want to choose a high privacy contract.
%, even if they are not highly concerned about security. 
%This happens because Eq. (\ref{eq:IC2_risk}) might not be satisfied when the contract ignores loss risk.

Intuitively, the utility company might need to decrease the cost $t_L$ and/or
increase the privacy setting $x_L$  in order to promote participation of
consumers.
These measures might decrease the benefit and fees collected by the utility
company. Hence, there is an incentive for the utility company to purchase insurance and/or invest in security.
%privacy risk might reduce the profit of the utility and 
%the utility company 
%they might have to 
%implement a mechanism to reduce the risk and promote the participation of users, e.g., purchase insurance and/or invest in security.
%Note that the utility might not know the exact value of $\eta(x)$. Hence, contracts can be designed using  boundaries of $\eta$ to be sure about expected outcomes.
%In the next section we analyze the effects on the contracts caused when we
%consider loss risk.

In the rest of this section, we characterize the contracts with risk loss and
the utility company's profit.

\subsection{Contracts with Loss Risk}

%Note that the solution of the optimal contract in Section \ref{sec:contracts}
%applies to the case with risk as long as $U(\cdot)$ satisfies Assumption \ref{ass:U_concave}.
Suppose that $U$ defined in \eqref{eq:U_risk} satisfies Assumption
\ref{ass:U_concave}. Then the analysis in Section~\ref{sec:contracts} holds when
we replace $\hU$ with $U$.
Thus, we are going to compare the settings of contracts with and without risk, denoted as $\{x_L, x_H, t_L, t_H\}$, and $\{\hat{x}_L, \hat{x}_H, \hat{t}_L, \hat{t}_H\}$, respectively.

%In Assumption \ref{as:1} we consider the nominal utility function $\hat{U}(x, \theta)$ is increasing concave with respect to $x$. 
First, let us extract some inequalities that are going to be used to compare the contracts. 
From \eqref{eq:U_risk} we can extract the marginal utility with loss risk:
\begin{equation}\label{eq:marginal_utility}
\frac{\partial}{\partial x} U(x, \theta) = \frac{\partial}{\partial x} \hat{U}(x, \theta) + \frac{\partial}{\partial x} \eta(x) \ell(\theta).
\end{equation}
Since $U$ is strictly increasing by Assumption~\ref{ass:U_concave}, we have
%From the formulation we know that the probability of a successful attack ($1-\eta(x))$ decreases with higher privacy settings. Hence, 
%%
%\begin{equation}\label{eq:deriv_prob}
%\frac{\partial}{\partial x} \eta(x) > 0.
%\end{equation}
%%
%Since $\ell(\theta) > 0$ and $\hU(x, \theta)$ is increasing in $x$, we can use
%(\ref{eq:marginal_utility}) and (\ref{eq:deriv_prob}) to show that $U(x,
%\theta)$ is increasing; in particular,
\begin{equation}
\frac{\partial}{\partial x} U(x, \theta) > 0.
\end{equation}
Furthermore, 
since $\ell(\theta)\geq 0$ and the probability of a successful attack ($1-\eta(x))$ decreases with higher
privacy settings, we have
\begin{equation}\label{eq:deriv_prob}
\frac{\partial}{\partial x} \eta(x) > 0.
\end{equation}
Hence,
we can extract the following inequality from \eqref{eq:marginal_utility}
\begin{equation}\label{eq:ineq1}
 \frac{\partial}{\partial x} U(x,\thetaH) \geq \frac{\partial}{\partial x} \hat{U}(x,\thetaH).
\end{equation}
From this inequality we can infer that the presence of risk increases the \emph{valuation} that each user gives to its privacy.

Now, we proceed to analyze changes in the optimal contract with the inclusion of
loss risk. First, let us show that the privacy setting of a user with high-type is greater in presence of risk. 
%This result is stated in the following proposition.

\begin{proposition}\label{prop:x_H}
 The privacy policy of an agent with type $\thetaH$ is higher with loss risk, that is, $x_H^* \geq \hat{x}_H^*$.
\end{proposition}
\begin{proof}
%Let us consider a part of the optimal contract (see \eqref{eq:P3a}) to extract the following first-order conditions (FOC) for the case without risk
%
From \eqref{eq:P3a}, we have the first-order conditions for the
case without risk,
\begin{equation}\label{eq:foc2_a}
\frac{\partial}{\partial x} \left( \hat{U}(x,\thetaH)  - g(x) \right) \Big|_{x=\hat{x}_H^*}  = 0,
\end{equation}
and the first-order conditions for the case with risk,
\begin{equation}\label{eq:foc2_b}
\frac{\partial}{\partial x} \left( U(x,\thetaH)  - g(x) \right) \Big|_{x=x_H^*}  = 0.
\end{equation}
%
%
%We can replace \eqref{eq:ineq1} in Eq. (\ref{eq:foc2_b}) to obtain 
Thus \eqref{eq:ineq1} implies
\begin{equation}
0 \geq \frac{\partial}{\partial x} \left( \hat{U}(x,\thetaH)  - g(x) \right) \Big|_{x=x_H^*}.
\end{equation}
Note that $\frac{\partial}{\partial x} \left( \hat{U}(x,\thetaH)  - g(x)
\right)$ is a decreasing function of $x$. Hence, the optimal privacy setting
without risk $\hat{x}_H^*$ (which satisfies \eqref{eq:foc2_a}) must be smaller than the privacy setting with risk, i.e.,    $x_H^* \geq \hat{x}_H^*$. This result is independent of the population distribution $p$ and the low type  $\thetaL$. 
\end{proof}

%\begin{proposition}
%  The price for an agent of type $\theta_H$ is characterized by
%  \begin{equation}
%    \left\{\begin{array}{cc}
%      t_H>\hat{t}_H, & \ \ \text{if }\ \ 1-\frac{1-\eta(x_H)
%    \label{<++>}
%  \end{equation}<++>
%\end{proposition}<++>

Now, we analyze the privacy settings for consumers with low-type. 
%In the following proposition we present two possible scenarios for $x_L^*$ in function of the population proportion $p$ and the users type.

\begin{proposition}\label{prop:x_L}
The privacy of low-type agents in contracts with and without loss risk ($x_L^*$
and $\hat{x}_L^*$ resp.) satisfy the following inequalities:
\begin{equation}
 \left\{
 \begin{array}{ l l }
%   x_L^* = \hat{x}_L^* & \text{if } p = \bar{p}, \\
  x_L^* \geq \hat{x}_L^*, & \text{if } p \leq \bar{p}, \\
  x_L^* < \hat{x}_L^*, & \text{if } p > \bar{p}, \\
 \end{array}
 \right.
\end{equation}
with $\bar{p} = \frac{\ell(\thetaL)}{\ell(\thetaH)}$.%^{\ell(\thetaL)} /_ {\ell(\thetaH)} $.

\end{proposition}
\begin{proof}
%Let us consider Eq. (\ref{eq:P3b}) (maximization problem with respect to $x_L^*$) to extract the following FOC for the case without risk
%
From \eqref{eq:P3b} we get the first-order conditions for the case without risk,
\begin{equation}\label{eq:foc1}
\frac{\partial}{\partial x} \left( \hat{U}(x,\thetaL) - p \hat{U}(x,\thetaH) \right)  \Big|_{x = \hat{x}_L^*} - (1-p) \frac{\partial}{\partial x} g(x) \Big|_{x = \hat{x}_L^*} = 0,
\end{equation}
and the first-order conditions for the case with loss risk,
\begin{equation}\label{eq:foc2}
\frac{\partial}{\partial x} \left( U(x,\thetaL) - p U(x,\thetaH) \right)  \Big|_{x = x_L^*} - (1-p) \frac{\partial}{\partial x} g(x) \Big|_{x = x_L^*} = 0.
\end{equation}
We use \eqref{eq:U_risk} to rewrite the first term of \eqref{eq:foc2} as
%Using \eqref{eq:U_risk} %in \eqref{eq:foc2}, we have
%
\begin{align}\label{eq:equivalence}
\frac{\partial}{\partial x} \left( U(x,\thetaL) - p U(x,\thetaH) \right)  
%\Big|_{x = x_L^*} 
&= 
\frac{\partial}{\partial x} \left( \hat{U}(x,\thetaL) - p \hat{U}(x,\thetaH)
\right)\notag  \\ 
 &\quad+\frac{\partial}{\partial x} \eta(x) \left( \ell(\thetaL) - p \ell(\thetaH)
\right)  .
%\Big|_{x = x_L^*} 
\end{align}

Now, let us consider three cases. First, if $ \ell(\thetaL) - p \ell(\thetaH) =
0$, then, from \eqref{eq:equivalence} we know that the contracts with and without risk have the same solution, that is, $x_L^* = \hat{x}_L^*$.

If $ \ell(\thetaL) - p \ell(\thetaH) > 0$, then 
\begin{equation}\label{eq:ineq_case_1}
\frac{\partial}{\partial x} \left( U(x,\thetaL) - p U(x,\thetaH)  \right) 
> 
\frac{\partial}{\partial x} \left( \hat{U}(x,\thetaL) - p \hat{U}(x,\thetaH)  \right).
\end{equation}
%
%If we replace Eq. (\ref{eq:ineq_case_1}) into Eq. (\ref{eq:foc2}) we obtain 
Then \eqref{eq:ineq_case_1} and \eqref{eq:foc2} imply
\begin{equation}
  \label{eq:neg} 0 > \frac{\partial}{\partial x} \left( \hat{U}(x,\thetaL) - p \hat{U}(x,\thetaH) - (1-p) \frac{\partial}{\partial x} g(x) \right) \Bigg|_{x = x_L^*}.
\end{equation}
%Hence, we know that the FOC for the case without risk is satisfied for some privacy $\hat{x}_L^*$ lower than the optimal privacy of the case with risk.
%In other words, since the maximization problem in Eq. (\ref{eq:P3b}) is convex, the previous expression indicates that  $x_L^* > \hat{x}_L^*$.
Hence, by \eqref{eq:foc1}, $x_L^\ast>\hx_L^\ast$.
% We can use the FOC in Eqs. (\ref{eq:foc1}) and (\ref{eq:foc2}), the inequality in Eq. (\ref{eq:inex_case_1}), and the fact that $\frac{\partial}{\partial x} g(x)$ is increasing to prove that loss risk increase the privacy of low type users, i.e.,  $x_L^* > \hat{x}_L^*$. 

Finally, if $ \ell(\thetaL) - p \ell(\thetaH) < 0$, then 
\begin{equation}\label{eq:ineq_case_2}
\frac{\partial}{\partial x} \left( U(x,\thetaL) - p U(x,\thetaH)  \right) 
< 
\frac{\partial}{\partial x} \left( \hat{U}(x,\thetaL) - p \hat{U}(x,\thetaH)  \right).
\end{equation}
In this case, we can use the first-order conditions in \eqref{eq:foc1} and \eqref{eq:foc2} and the
inequality \eqref{eq:ineq_case_2} to prove that  $x_L^* < \hat{x}_L^*$.
\end{proof}

%\subsection{Contract as a function of the population distribution $p$}

The following result generalizes the dependence of the privacy $x_L$ with
respect to $p$. In particular, the privacy of the low-type users is decreasing with respect to $p$, regardless of the presence of loss risk.

\begin{proposition}
 The optimal  privacy setting $x_L^*$ is decreasing with respect to $p$.
\end{proposition}

\begin{proof}
First, let us consider two distribution probabilities $p, \hat{p} \in[0,1]$ such
that $\hat{p} = p + \epsilon$, where $\epsilon>0$. Now, let us define $x_L^*(p)$
as optimal privacy policy that satisfies the first-order conditions in \eqref{eq:foc2}, for a population distribution $p$.
Also, let us consider the derivative of the objective function in \eqref{eq:P3b} for a population distribution $\hat{p}$:
\begin{equation}
\frac{\partial}{\partial x} f(x, \hat{p}) = \frac{\partial}{\partial x} \left( U(x, \thetaL) - \hat{p} U(x, \thetaH) - (1-\hat{p}) g(x) \right),
\end{equation}
where $f(x,\hat{p})$ is the objective function of the optimization problem.
This can be rewritten as
\begin{align}\label{eq:foc3}
\frac{\partial}{\partial x} f(x, \hat{p}) &= \frac{\partial}{\partial x} \left(
U(x, \thetaL) - p U(x, \thetaH) - (1-p) g(x) \right) \notag \\ 
 &\quad+ \epsilon \frac{\partial}{\partial x} \left( g(x) - U(x, \thetaH) \right).
\end{align}
Evaluating the previous equation in $x_L^*(p)$ we find that
\begin{equation}\label{eq:inex_p}
\frac{\partial}{\partial x} f(x_L^*(p), \hat{p})= \epsilon \frac{\partial}{\partial x} \left( g(x) - U(x, \thetaH) \right) \Big|_{x = x_L^*(p)} .
\end{equation}
The previous result follows since $x_L^*(p)$ satisfies the first-order conditions in \eqref{eq:foc2}.

Now, recall that the contract assigns higher privacy to agents with high-type,
i.e.,  $x_H^*>x_L^*$. Hence, if we evaluate the left-hand side of \eqref{eq:foc2_b} in $x_L^*(p)$, we find that 
\begin{equation}\label{eq:ineq3}
\frac{\partial}{\partial x} \left( U(x, \thetaH) - g(x)  \right) \Big|_{x = x_L^*(p)} > 0.
\end{equation}
%
%Thus, it is reasonable to assume that $\frac{\partial}{\partial x} U(x, \thetaH) > \frac{\partial}{\partial x} g(x)$. 
Thus, we know that 
\begin{equation}
\frac{\partial}{\partial x} f(x_L^*(p), \hat{p}) < 0.
\end{equation}
This equation indicates that the the optimal contract with a distribution $\hat{p}$ satisfies
%$$
$x_L^*(p) > x_L^*(\hat{p})$.
%$$
\end{proof}

The previous result states that the optimal privacy for low-type agents $x_L^*$
is decreasing with respect to the population distribution $p$. Consequently,
$t_L^*$ is also decreasing with $p$. Furthermore, $x_H^*$ does not depend on $p$
since the high-type always gets an efficient allocation independent of the prior
on types.
%, but we can prove that  $t_H^*$ is increasing with respect to $p$.
This applies regardless of the risk probability $1-\eta(x)$.

Another aspect of the contract that changes with the introduction of risk is the cost $t$.
%of joining the DLC program.
Results in Propositions \ref{prop:x_H} and \ref{prop:x_L} let us determine the impact of risk on the price $t$ paid by each user. 
%The following proposition summarizes the relation of $t$ with the 

\begin{proposition}
The price of the contracts with and without risk ($\{t_L^*, t_H^*\}$ and $\{\hat{t}_L^*, \hat{t}_H^*\}$ respectively) satisfy the following inequalities:
\begin{equation}
\left\{ 
\begin{array}{ll}
  t_L^* > \hat{t}_L^*-(1-\eta(\hat{x}_L^\ast))\ell(\theta_L), & \text{if } p < \bar{p} \\
t_L^* < \hat{t}_L^*, & \text{if } p > \bar{p} \\
t_H^* > \hat{t}_H^*, & \text{if } p>\bar{p},
1-\frac{1-\eta(x_H^\ast)}{1-\eta(x_L^\ast)}>\bar{p}\end{array}
\right.
\end{equation}
with $\bar{p} = \frac{\ell(\thetaL)}{\ell(\thetaH)} $.
\end{proposition}
\begin{proof}
From \eqref{eq:newconst2} we know that the contract payment is
\begin{equation}\label{eq:price_l1}
t_L^* = U(x_L^*, \thetaL).
\end{equation}
Since $U(\cdot, \theta)$ is an increasing function in $x$, we can use Proposition \ref{prop:x_L} to conclude that
\begin{equation}
\left\{ 
\begin{array}{ll}
t_L^* > \hat{t}_L^*-(1-\eta(\hat{x}_L^\ast))\ell(\theta_L), & \text{if } p < \bar{p} \\
t_L^* < \hat{t}_L^*, & \text{if } p > \bar{p} \\
\end{array}
\right.
\end{equation}
On the other hand, the price paid by high-type users is given by
%can be defined as (see Eqs. (\ref{eq:newconst1}) and (\ref{eq:newconst2}))
\begin{equation}\label{eq:price_h1}
t_H^* = U(x_H^*, \thetaH) - U(x_L^*, \thetaH) + U(x_L^*, \thetaL).
\end{equation}
% replace Eq. (\ref{eq:price_l1}) in (\ref{eq:price_h1}) and
%Now, let us  evaluate the difference between $t_H^*$ and $\hat{t}_H^*$:
%%
%\begin{align}\label{eq:diff1}
%t_H^* - \hat{t}_H^* &= 
%U(x_H^*, \thetaH) - \hU(\hat{x}_H^*, \thetaH) \notag\\
%&\quad+  U(x_L^*, \thetaL)  - U(x_L^*, \thetaH)  - 
%\left(  \hU(\hat{x}_L^*, \thetaL) - \hU(\hat{x}_L^*, \thetaH) \right)
%\end{align}
%
Hence,
\begin{align}
  t_H^\ast=&\hU(x_H^\ast, \thetaH)-(1-\eta(x_H^\ast))\ell(\thetaH)-\hU(x_L^\ast,
  \thetaH)\notag\\
  &+(1-\eta(x_L^\ast))\ell(\thetaH) +\hU(x_L^\ast, \thetaL)-(1-\eta(x_L^\ast))\ell(\thetaL)
  \label{eq:thast}
\end{align}
By assumption $\hU(\cdot, \theta)$ is increasing in $x$ so that $\hU(x_H^\ast,
\thetaH)>\hU(\hat{x}_H^\ast, \thetaH)$ since $x_H^\ast>\hat{x}_H^\ast$ for all
$p$ by Proposition \ref{prop:x_H}. 
%The following inequality arises from Proposition \ref{prop:x_H} and the properties of $U$:
%\begin{equation}
%U(x_H^*, \thetaH) > U(\hat{x}_H^*, \thetaH), \ \ \forall \ p.
%\end{equation}
%%for all $p$. 
Furthermore, if we define 
\begin{equation}
h(x) = -( \hU(x, \thetaH) - \hU(x, \thetaL) ),
\end{equation}
then from Assumption \ref{ass:marginal} we know that $h(\cdot)$ is decreasing in $x$.
Hence, by Proposition 2, for $p>\bar{p}$ we have $x_L^\ast<\hat{x}_L^\ast$ so
that \eqref{eq:thast} becomes
\begin{align}
  t_H^\ast&>\hU(\hat{x}_H^\ast,
  \thetaH)-(1-\eta(x_H^\ast))\ell(\thetaH)-\hU(\hat{x}_L^\ast,
  \thetaH)\notag\\
  &\ \ +(1-\eta(x_L^\ast))\ell(\thetaH) +\hU(\hat{x}_L^\ast, \thetaL)-(1-\eta(x_L^\ast))\ell(\thetaL)\\
  &=\hat{t}_H^\ast+(\eta(x_H^\ast)-\eta(x_L^\ast))\ell(\thetaH)-(1-\eta(x_L^\ast))\ell(\thetaL)
  \label{eq:thatast}
\end{align}
%\begin{equation}
% t_H^* > \hat{t}_H^*
% +(\eta(x_H^\ast)-\eta(x_L^\ast))\ell(\thetaH)-(1-\eta(x_L^\ast))\ell(\thetaL).  
%\end{equation}
By assumption $1-\frac{1-\eta(x_H^\ast)}{1-\eta(x_L^\ast)}>\bar{p}$, so that
$t_H^\ast>\hat{t}_H^\ast$.

%
%\begin{equation}
%h(x_L^*) - h(\hat{x}_L^*) > 0
%\end{equation}
%when $x_L^* < \hat{x}_L^*$.
% Note that $h(x_L^*) - h(\hat{x}_L^*)$ is equivalent to
%the right-hand side of \eqref{eq:diff1}. Thus,  by Proposition~\ref{prop:x_L}, if $p>\bar{p}$ then $x_L^* < \hat{x}_L^*$ so that
%%we know that 
%\begin{equation}
% t_H^* > \hat{t}_H^*
% +(\eta(x_H^\ast)-\eta(x_L^\ast))\ell(\thetaH)-(1-\eta(x_L^\ast))\ell(\thetaL).  
%\end{equation}
%Hence, if $1-\frac{1-\eta(x_H^\ast)}{1-\eta(x_L^\ast)}>\bar{p}$, we have that
%$t_H^\ast>\hat{t}_H^\ast$.
%and hence, the result.
%if $x_L^* < \hat{x}_L^*$, which holds when $p>\bar{p}$. 
%However, if $p<\bar{p}$  we cannot extract a relation between $t_H^*$ and $\hat{t}_H^*$ without making additional assumptions.
\end{proof}

\begin{proposition}
\label{prop:inforent}
If $p>\bar{p}$, then the information rent under the optimal contract without loss risk is higher than under the optimal contract with loss risk, i.e.
\begin{equation}
\hU(\hx_L^\ast, \thetaH)-\hU(\hx_L^\ast, \thetaL)>U(x_L^\ast, \thetaH)-U(x_L^\ast, \thetaL).
\end{equation}
\end{proposition}
\begin{proof}
Suppose that $p>\bar{p}$, then by Proposition~\ref{prop:x_L} we have $x_L^\ast<\hat{x}_L^\ast$. The information rent under $\{x_L^\ast, x_H^\ast, t_L^\ast, t_H^\ast\}$ is given by
\begin{align}
U(x_L^\ast, \thetaH)-U(x_L^\ast, \thetaL)=&\hU(x_L^\ast, \thetaH)-\hU(x_L^\ast, \thetaL)\notag\\
&+(1-\eta(x_L^\ast))(\ell(\thetaL)-\ell(\thetaH)).
\end{align}
Since $\ell(\thetaL)<\ell(\thetaH)$ by assumption, we have
\begin{align}
U(x_L^\ast, \thetaH)-U(x_L^\ast, \thetaL)<\hU(x_L^\ast, \thetaH)-\hU(x_L^\ast, \thetaL)
\end{align}
By Assumption~\ref{ass:U_concave}, $\hU$ is increasing in $x$. Hence,
\begin{align}
\hU(x_L^\ast, \thetaH)-\hU(x_L^\ast, \thetaL)<\hU(\hx_L^\ast, \thetaH)-\hU(\hx_L^\ast, \thetaL).
\end{align}
Thus,
\begin{align}
U(x_L^\ast, \thetaH)-U(x_L^\ast, \thetaL)<\hU(\hx_L^\ast, \thetaH)-\hU(\hx_L^\ast, \thetaL).
\end{align}
\end{proof}
We can conclude that when $ p < \bar{p}$,  the loss risk increases the privacy
contract of all users. Roughly speaking, this happens because losses of users
with low-type are significant with respect to losses of high-type users. In
consequence, the contract with loss risk allows users with low-type to have more
privacy.
%to reduce risk costs. 
%However, low-type users have to make higher payments to have higher privacy, i.e., the price $t_L^*$ is greater with loss risk. 
On the other hand, when $p > \bar{p}$ losses of low-type users are not
significant and the contract will tend to offer less privacy settings to
low-type users. 
In this case the utility company can afford more losses due to risk in order to
collect higher fidelity data.

We can conclude that, regardless of the profit, a population of agents with $ p
> \bar{p}$ might be more beneficial for a utility company interested in collecting information from users.
%Note that $x_H^*$ depends on the loss risk, but is not affected by $p$ or $\thetaL$.
Roughly speaking, a favorable environment for privacy contracts is characterized
by a large population of agents with high-type. 
In the next section we analyze the contract parameters and the utility company's profit as a function of the population distribution $p$.

%Now, if $p = \frac{\ell(\thetaL)}{\ell(\thetaH)}$, then the privacy contract is equal to the case without risk, that is,  $x_L^*=\hat{x}_L^*$. 

\iffalse
However, if 
%
\begin{equation}
\frac{\partial}{\partial x} \left( g(x) - U(x, \thetaH) \right) \Big|_{x = x_L^*} > 0,
\end{equation}
%
then $x_L^*$ and $t_L^*$ are increasing with $p$. On the other hand, $t_H^*$ is decreasing in function of $p$.
%
\fi

%%%%% profit
\subsection{Profit}

The profit of the utility company is
\begin{align}\label{eq:profit}
\Pi(t_L, x_L, t_H, x_H)& = (1-p) \left( -g(x_L) + U(x_L, \thetaL) \right)\notag  \\
&\quad + p \left( -g(x_H) + t_H \right).
\end{align}
Recall that $x_H$ does not depend on $p$ and that 
$t_H$ is increasing with respect to $p$. Hence, $p \left( -g(x_H) + t_H \right)$ is increasing in $p$. 
Also, note that \eqref{eq:foc2} can be rewritten as
\begin{equation}\label{eq:inex_marginal}
\frac{\partial}{\partial x} \left( U(x, \thetaL) - g(x) \right)\Big|_{x = x_L^*(p)} = p 
\frac{\partial}{\partial x} \left( U(x, \thetaH) - g(x) \right)\Big|_{x = x_L^*(p)}
\end{equation}
From \eqref{eq:ineq3} we know that $\frac{\partial}{\partial x} \left( U(x,
\thetaH) - g(x) \right) \big|_{x = x_L^*(p)} $. Therefore,   $U(x_L^*(p),
\thetaH) - g(x_L^*(p))$ is increasing in $x_L^*(p)$.
%(Note that a function is increasing if its first derivative is positive). 
%
Also, because $p\geq 0$, we know that $U(x_L^*(p), \thetaL) - g(x_L^*(p))$ is increasing in $x$. 
Considering that $x_L^*(p)$ is decreasing in $p$, 
we can conclude that $(1-p)(-g(x_L^*(p)) + U(x_L^*(p)), \thetaL)$ is decreasing in $p$. 

Thus, $\Pi$ is composed of an increasing term and a decreasing term
of %can exhibit both decreasing and increasing properties as function of
$p$ so that the maximum profit might be achieved on the boundary, i.e. 
$p=0$ or $p=1$.
We explore the profit of the utility company in more detail through an example in the following section.
%at a extreme point. That is, the maximum profit can be achieved when the population distribution is either $p=0$ or $p=1$. 
%However, i think this might be subject to the form of each component of $\Pi$, e.g., concavity, convexity or things like that

%Note that as $p$ increases, the optimal policy $x_L^*$ might be negative. In this case the utility cannot offer contracts to low privacy agents, because there is a limit in the privacy that each agent can provide.

\section{Example -- Direct Load Control}
\label{sec:dlc}
%We now show the effects of risk through an example.

Recall that the unit gain the utility company gets out of the privacy setting $\q$ is a
function $g:\mc{X}\rar\mb{R}$. In this section, we discuss a particular example
in which $g$ is a metric for how access to high-granularity data affects direct
load control (DLC). 

In previous work, we characterized the utility--privacy tradeoff
for a DLC problem of thermostatically controlled
loads (TCLs)~\cite{dong:2014ab}. In particular, we showed that the $\ell_1$--norm of
the error of the DLC (measured in terms of the $\ell_1$ distance between the
actual power consumed by the TCLs and the desired power consumption) increases
as a function of the sampling period, i.e. distance between samples. Empirically
the relationship between sampling
period and $\ell_1$--norm error is approximately quadratic.
%in the particular example for
%DLC of TCLs. 
Hence, for this example, we let
%such that $g(\q_L)>g(\q_H)$ and decreases in a linear way. Hence, for this example, let 
\begin{equation}
  g(\q)=\frac{1}{2}\zeta \q^2
  \label{eq:gxdlc}
\end{equation}
where $0<\zeta<\infty$.

Note that a decreased sampling rate corresponds to a higher privacy setting. The
function $g$ as defined is increasing in $\q$ so that $g(\q_L)<g(\q_H)$. %Assume that the consumer's utility is proportional to the probability of an adversary succeeding $1-\eta$, how close they are to the maximum privacy setting $\bar{q}$, and their type, i.e. $U(\q, \theta)=-\frac{1}{2}(1-\eta)(\q-\bar{\q})^2\theta$ where $q\in [0, \bar{q}]$. Thus, 

%Let us assume that the utility company stores consumer data in a database and $1-\eta$ is the probability that an adversary will successfully implement an attack on that database. 
Suppose that
$\theta\in\{\thetaL, \thetaH\}$ where $0<\thetaL<\thetaH$. Assume that the consumer's utility is 
given by
\begin{equation}
  \hat{U}(\q, \theta)=\q\theta%\frac{1}{2}(\bar{\q}^2-(\q-\bar{\q})^2)\theta
  \label{eq:agentutility}
\end{equation}
where $\q\in [0, 1]$, i.e. their utility is proportional to both the type
and quality. %so that it is proportional to how close they are to
%the maximum privacy setting $\bar{\q}$, and their type $\theta$. 
As in the previous section, we model the consumer's risk aversion using the
following utility:
\begin{equation}
  U(x, \theta)=\hU(x,\theta)-(1-\eta(x))\ell(\theta)
  \label{eq:riskaver}
\end{equation}

%It is known that
%$\q'\geq \q$ that $1-\eta(\q')\leq 1-\eta(\q)$, i.e. 
A higher privacy setting is
less likely to be successfully attacked; hence, for the sake of the example, we
take $1-\eta(x)=m(1-x)$ where $m>0$ is a constant and $x\in [0,1]$ with zero corresponding to a low-privacy
setting and one corresponding to a high-privacy setting.
%We explore different values of $m$. 

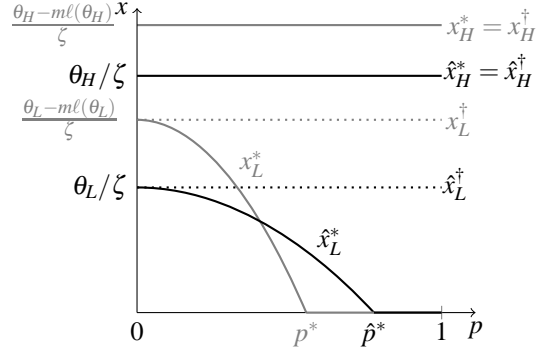
\begin{figure}[ht]
  \begin{center}
    \begin{tikzpicture}[scale=0.9]
\draw[->] (0,0) -- (5,0) node[anchor=north] {$p$};
\draw	(0,0) node[anchor=north] {0}
(2.5,0) node[anchor=north,gray] {$p^\ast$}
		(4.5,0) node[anchor=north] {1};
\draw[gray] (4.5,0.1) -- (4.5, -0.1);
\node[gray] at (4.75,2.85) {$\q_L^\dagger$};
\node[gray] at (1.7, 2.2) {$\q_L^\ast$};
\node[gray] at (5.25, 4.25) {$\q_H^\ast=\q_H^\dagger$};
\draw[thick,gray] (0,4.25) -- (4.5, 4.25);
\draw[dotted, thick,gray] (0,2.85) -- (4.5, 2.85);
\draw[->] (0,0) -- (0,4.5) node[anchor=east] {$\q$};
\node[gray] at (-1.1, 4.25) {$\frac{\thetaH-m\ell(\thetaH)}{\zeta}$};
\node[gray] at (-1, 2.85) {$\frac{\thetaL-m\ell(\thetaL)}{\zeta}$};
\draw[thick,gray] (0,2.85) parabola (2.5,0);
\draw[thick] (0, 1.85) parabola (3.5,0);
\node[] at (2.85,1.05) {$\hat{\q}_L^\ast$}; 
\draw[thick,gray] (2.5, 0) -- (4.5, 0);
\draw[thick] (0, 3.5) -- (4.5, 3.5);
\draw[thick, dotted] (0, 1.85) -- (4.5, 1.85);
\node[] at (4.7, 1.9) {$\hat{x}_L^\dagger$};
\node[] at (5.2, 3.575) {$\hat{x}_H^\ast=\hat{x}_H^\dagger$};
\draw[thick] (3.5,0) -- (4.5,0);
\node[below] at (3.5,0) {$\hat{p}^\ast$};
\node[left] at (0, 1.85) {$\thetaL/\zeta$};
\node[left] at (0, 3.5) {$\thetaH/\zeta$};

\end{tikzpicture}  
\end{center}
  \caption{Comparison between full information and asymmetric information
  solutions as a function of $p$ the probability of the high-type in the
  population for both the case when we consider risk (gray) and when we do not
  consider risk (black). The general shape of the curves stay the same for
  different values of $m$; changing $m$ from $0$ to $1$ only has an effect of shifting the
  critical value $p^\ast$ for the case with risk closer to the origin as well as causing
  $x_H^\ast$ to decrease.}
  \label{fig:info}
\end{figure}
%If the utility company knows the true type of the consumer, the we call the
% solution the \emph{first-best} solution. 
 For the case without risk the
 first-best solution is given by
 \begin{equation}
   (\hat{\q}_H^\dagger, \hat{\q}_H^\dagger)=\left( \frac{\thetaH}{\zeta},
   \frac{\thetaL}{\zeta}
   \right) \ \text{and} \ (\hat{t}_L^\dagger, \hat{t}_H^\dagger)=\left(
   \frac{\thetaH^2}{\zeta}, \frac{\thetaL^2}{\zeta} \right).
   \label{eq:FBwoR}
 \end{equation}
 For the case with risk, the first-best solution is given by
 \begin{equation}
   (\q_H^\dagger, \q_L^\dagger)=\left( \frac{\thetaH-m\ell(\thetaH)}{\zeta},
   \frac{\thetaL-m\ell(\thetaL)}{\zeta} \right)
   \label{eq:FBwr}
 \end{equation}
and
\begin{equation}
  (t_H^\dagger, t_L^\dagger)=\left(\frac{\thetaH^2-m\thetaH\ell(\thetaH)}{\zeta},
  \frac{\thetaL^2-m\thetaL\ell(\thetaL)}{\zeta}\right).
  \label{eq:FBwrt}
\end{equation}
%The solution when the utility company does not know the type of the consumer but
%instead knows a prior on the types in the population is called the
%\emph{second-best} solution. 
Let the utility company's prior be defined by
$p=\text{P}(\theta=\thetaH)$ and $(1-p)=P(\theta=\thetaL)$. Then the optimal solution to the screening
problem without risk is given by
\begin{equation}
  (\hat{\q}_H^\ast, \hat{\q}_L^\ast)=\left(
  \frac{\thetaH}{\zeta},\frac{1}{\zeta}\left[\thetaL-\frac{p}{1-p}(\thetaH-\thetaL)\right]_+\right)
  %\frac{\bar{\q}(p\thetaH-\thetaL)}{p\thetaH-\thetaL-\zeta+p\zeta} \right]_+ \right)
  %(\q_H^\ast, \q_L^\ast)=\left( \bar{\q}-\frac{\zeta}{\thetaH}, \left[\bar{\q}+\frac{(1-p)\zeta}{(p\thetaH-\thetaL)}\right]_+ \right)
%  (x_H^\ast, x_L^\ast)=\left(c-\frac{1}{\thetaH},c-\frac{1-p}{\thetaL-p\thetaH}\right)
  \label{eq:optH}
\end{equation}
where $[\cdot]_+=\max\{\cdot, 0\}$, i.e. 
 $\hat{\q}_L^\ast=0$ when $p\geq \hat{p}^\ast$, and
 \begin{align}
   \hat{t}_H^\ast&=\frac{\thetaH^2}{\zeta}-\frac{(\thetaL-\thetaH)}{\zeta}\left[\thetaL-\frac{p}{1-p}(\thetaH-\thetaL)\right]_+\\
  %\frac{1}{2}\left(
  %\frac{\thetaH^2\bar{\q}^2(\thetaH-2\zeta)}{(\thetaH-\zeta)^2} \right)\\
  \hat{t}_L^\ast&=\frac{\thetaL}{\zeta}\left[\thetaL-\frac{p}{1-p}(\thetaH-\thetaL)\right]_+
%  
%  \left[ \frac{1}{2}\thetaL\left( \bar{\q}^2-\left(
%  \frac{\bar{\q}(p\thetaH-\thetaL)}{p\thetaH-\thetaL-\zeta+p\zeta}-\bar{\q}
%  \right)^2 \right) \right]_+
  \label{eq:actualoptts}
\end{align}

Similarly, the optimal solution to the screening problem when there is risk is
given by
\begin{align}
  (x^\ast_H, x_L^\ast)=\Big( &\frac{1}{\zeta}(\thetaH+m\ell(\thetaH)),\notag\\
 & \frac{1}{\zeta}\Big[
  \frac{m\ell(\thetaL)-pm\ell(\thetaH)-p\thetaH+\thetaL}{1-p}
 % m\ell(\thetaL)+\thetaL\notag\\
 % &+\frac{p}{1-p}\left(
 % m(\ell(\thetaL)-\ell(\thetaH))+\thetaL-\thetaH
 % \right) 
  \Big]_+ \Big)
  \label{eq:optRisk}
\end{align}
and
\begin{align}
  t_H^\ast&=t_L^\ast+x_H^\ast\thetaH-x_L^\ast\thetaH+m(x_H^\ast-x_L^\ast)\ell(\thetaH)\\
  t_L^\ast&=\frac{\thetaL}{\zeta}\Big[
  \frac{m\ell(\thetaL)-pm\ell(\thetaH)-p\thetaH+\thetaL}{1-p} \Big]_+ 
  \label{eq:ts}
\end{align}

The plots in Figures~\ref{fig:info}-\ref{fig:sw} show the fundamental properties
of the contract solution and the general shapes of the curves are invariant
under changes to parameters of the problem. %in $m$ and choices of the loss functions $\ell(\theta)$.

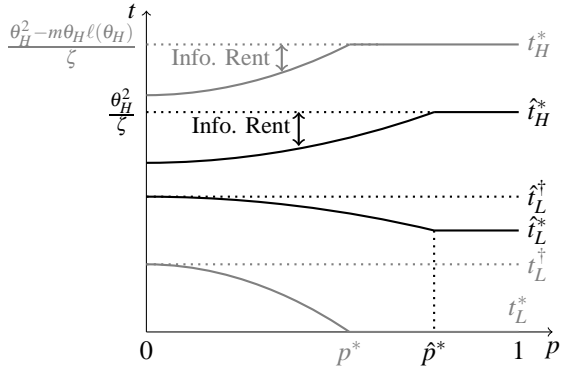
\begin{figure}[ht]
  \begin{center}
      \begin{tikzpicture}[scale=0.9]
    \draw[->] (0,0) -- (6,0) node[anchor=north] {$p$};
\draw[->] (0,0) -- (0,4.75) node[anchor=east] {$t$};
\draw[thick,gray] (0, 3.5) parabola (3.0, 4.25);
\draw[thick,gray] (3.0, 4.25) -- (5.5, 4.25);
\node[left,gray] at (0, 4.25) {$\frac{\thetaH^2-m\thetaH\ell(\thetaH)}{\zeta}$}; 
\draw[thick, dotted, gray] (0, 4.25) -- (3.5, 4.25);
\draw[<->, thick,gray] (2,4.25) -- (2,3.85);
\node[above, gray] at (1.1,3.8) {\small Info. Rent};
%\draw[thick, dotted, gray] (0,3.5) -- (4.5, 3.5);
\node[gray,right] at (5.5, 4.25) {$t_H^\ast$};

\draw[thick,gray] (0, 1) parabola (3.0, 0);
\node[below,gray] at (3.0,0) {$p^\ast$};
\draw[thick,gray] (3.0, 0) -- (5.5, 0);
\draw[thick, dotted, gray] (0,1) -- (5.5,1);
\node[right, gray] at (5.5, 1) {$t_L^\dagger$};
\node[above,gray] at (5.5, 0) {$t_L^\ast$};

\draw[thick] (0,2.0) parabola (4.25,1.5);
\node[below] at (4.25,0) {$\hat{p}^\ast$};
\draw[dotted, thick] (4.25,1.5) -- (4.25,0);
\draw[thick] (4.25,1.5) -- (5.5, 1.5);
\draw[thick, dotted] (0,2.0) -- (5.5,2.0);
\node[right] at (5.5,2.0) {$\hat{t}_L^\dagger$};
\node[right] at (5.5, 1.5) {$\hat{t}_L^\ast$};
\node[right] at (5.5, 3.25) {$\hat{t}_H^\ast$};
\draw[thick] (0,2.5) parabola (4.25,3.25);
\draw[thick] (4.25, 3.25) -- (5.5, 3.25);
\draw[thick, dotted] (0,3.25) -- (4.25, 3.25);
\draw[<->, thick] (2.25,3.25) -- (2.25,2.75);
\node[above] at (1.4,2.8) {\small Info. Rent};
\node[left] at (0, 3.25) {$\frac{\thetaH^2}{\zeta}$};%\frac{\thetaL^2-m\thetaL\ell(\thetaL)}{\zeta}$};
\node[below] at (0,0) {$0$};
\node[below] at (5.5, 0) {$1$};
  \end{tikzpicture}

  \end{center}
\caption{Optimal prices as a function of $p$ for both the case with risk (gray)
and without (black). The information rent as a function of $p$ for both cases is
also shown.}
  \label{fig:prices}
\end{figure}
 \begin{figure}[ht]
   \begin{center}
     \begin{tikzpicture}[scale=0.9]
       \draw[->] (0,0) -- (6,0) node[anchor=north] {$p$};
\draw[->] (0,0) -- (0,4.75) node[anchor=east] {$\Pi$};

%without risk
\draw[thick] (0,1) parabola (4,2);
\draw[thick] (4,2) -- (5.5,2.85);
\draw[thick, dotted] (4,2) -- (4,0);
\node[below] at (4,0) {$\hat{p}^\ast$};
\draw[thick, dashed] (0,1) -- (5.5, 2.85);
\node[left] at (0,1) {$\thetaL^2/(2\zeta)$};
\node at (4.75,1.75) {$\hat{\Pi}^\ast$};
\draw[->] (4.5, 1.9) -- (4.4, 2.2);
\node at (4.5, 3.25) {$\hat{\Pi}^\dagger$};
\draw[->] (4.5,3) -- (4.6, 2.7);

%with risk
\draw[gray, thick] (0,0.5) parabola (2.25,2.5);
\draw[gray, thick] (2.25,2.5) -- (3.5,4.5);
\draw[gray, thick, dotted] (2.25,2.5) -- (2.25,0);
\node[below, gray] at (2.25,0) {$p^\ast$};
\draw[thick, gray, dashed] (0,0.5) -- (3.5,4.5);
\node[left, gray] at (0,0.25) {$\frac{\thetaL^2-\ell^2(\thetaL)m^2}{2\zeta}$};
\node[gray] at (2,3.75) {$\Pi^\dagger$};
\draw[gray,->] (2,3.5) -- (2.1, 3);
\node[gray] at (3.5, 3) {$\Pi^\ast$};
\draw[gray,->] (3.25,3.25) -- (3,3.5);
\node[below] at (0,0) {$0$};
\node[below] at (5.5,0) {$1$};
     \end{tikzpicture}
   \end{center}
   \caption{Profit of the utility company as a function of $p$ for both the case
   with risk $\Pi(p)$ (gray) and without $\hat{\Pi}(p)$ (black).  }
   \label{fig:profit}
 \end{figure}
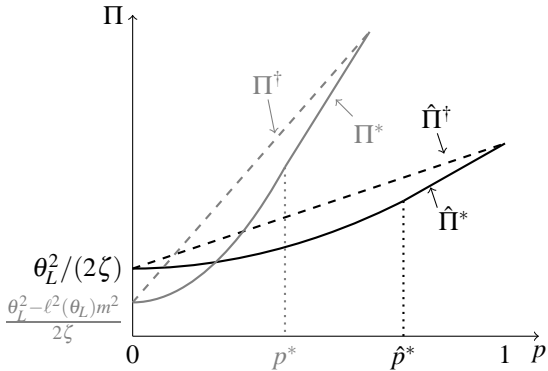
\begin{figure}[ht]
  \begin{center}
    \begin{tikzpicture}[scale=0.9]
      \draw[->] (0,0) -- (6,0) node[anchor=north] {$p$};
\draw[->] (0,0) -- (0,4.75) node[anchor=east] {$W$};
\draw[gray,thick] (2.25,1.25) parabola (0,0.5);
\draw[gray, thick] (2.25,1.25) -- (5.5,4.25);

\draw[thick] (4.0,1.8) parabola (0,1.2);
\draw[thick] (4.0,1.8) -- (5.5,2.8);
\draw[thick, dashed] (0,1.2) -- (5.5,2.8);
\node at (3.75,1.25) {$\hat{W}^\ast$};
\draw[->] (3.55,1.4) -- (3.35, 1.7);
\node at (5.2,1.8) {$\hat{W}^\dagger$};
\draw[->] (5.0, 1.9) -- (4.7, 2.5);

\draw[gray, thick, dashed] (0,0.5) -- (5.5,4.25);

\draw[dotted, thick] (5.5,3.0) -- (0,3.0);
\node[left] at (0,3.0) {$\frac{\thetaH^2}{2\zeta}$};
\draw[dotted, thick, gray] (5.5, 4.25) -- (0, 4.25);
\node[below, gray] at (2.25,0) {$p^\ast$};
\node[below] at (4.0,0) {$\hat{p}^\ast$};

\draw[dotted, thick, gray] (2.25,0) -- (2.25, 1.25);
\draw[dotted, thick] (4.0,0) -- (4.0, 1.8);

\node[gray] at (1,0.5) {$W^\ast$};
\draw[->,gray] (1,0.65) -- (0.9,0.95);
\node[gray] at (1.75,2.25) {$W^\dagger$};
\draw[->,gray] (1.9,2.15) -- (2.5, 2.3);

\node[below] at (5.5,0) {$1$};
\node[below] at (0,0) {$0$};
\draw[dotted, thick] (5.5,0) -- (5.5,4.25);
\node[left,gray] at (0,4.25) {$\frac{\thetaH^2-\ell^2(\thetaH)m^2}{2\zeta}$};
    \end{tikzpicture}
 \end{center}
  \caption{Social Welfare as a function of $p$ for both the case with risk
  $W^\ast$
  (gray) and without risk $\hat{W}^\ast$ (black). }
  \label{fig:sw}
\end{figure}
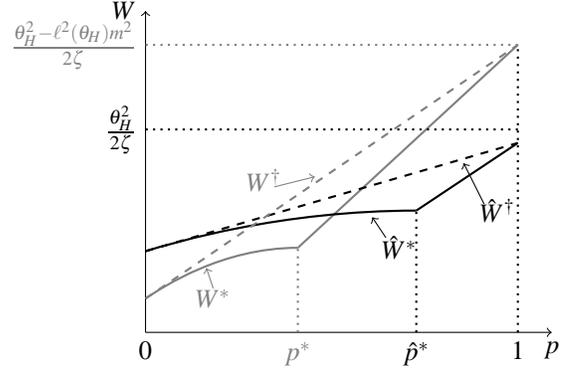

In Figure~\ref{fig:info}, we show that as the probability of the high-type being
drawn from the population increases, $\q_L^\ast$ (resp. $\hx_L^\ast$) decreases away from the optimal
full information solution $\q_L^\dagger$ (resp. $\hx_L^\dagger)$. This occurs until the critical point
\begin{equation}
  p^\ast=\frac{\thetaL+m\ell(\thetaL)}{\thetaH+m\ell(\thetaH)} \ \
  \left(\text{resp.}\ \hat{p}^\ast=\frac{\thetaL}{\thetaH}\right).
  \label{eq:critps}
\end{equation}
%
%$p^\ast$ (resp. $\hat{p}^\ast$).
It is reasonable that as soon as the probability of the utility company facing a consumer of high-type reaches a critical point, they will focus all their efforts on this type of consumer since a high-type desires a higher privacy setting which results in a degradation of the DLC scheme.

In Figure~\ref{fig:prices}, we show the optimal prices for the first- and
second-best problems in both the case with risk and without. We see
that for $p\leq p^\ast$ (resp. $p\leq \hat{p}^\ast$) we have positive
information rent for the high-type. Essentially, when the probability of the
existence of a 
low-type is large relative to the probability of the existence of a high-type,
there is a \emph{positive externality} on the high-type. Thus  people who value high privacy more need to be compensated more to
participate in the smart grid. Further, the low-type
continues to get zero surplus since the individual rationality constraint of the
low-type is always binding.

In Figure~\ref{fig:profit} and \ref{fig:sw} respectively, we show the utility
company's expected profit and social welfare %which is defined by %. % for
%%both the case with risk and without. 
%In addition, in Figure~\ref{fig:sw}, we show the social welfare,
\begin{align}
W(p,t_L, x_L, t_H, x_H)=&\Pi(t_L, x_L, t_H, x_H)+p(U(x_H, \thetaH)-t_H)\notag\\
&\quad+(1-p)(U(x_L, \thetaL)-t_L),
\label{eq:welfare}
\end{align}
which is the sum of expected profit of the utility company and the consumer.
%
%under both the first- and second-best solutions for the case with and without
%risk. 
%and with, i.e. $\hat{W}^\dagger=W(p,\hat{t}_L^\dagger, \hat{\q}_L^\dagger,\hat{t}_H^\dagger, \hat{\q}_H^\dagger)$, $\hat{W}^\ast=W(p,\hat{t}_L^\ast, \hat{\q}_L^\ast,\hat{t}_H^\ast, \hat{\q}_H^\ast)$, $W^\dagger=W(p,t_L^\dagger, x_L^\dagger, t_H^\dagger, x_H^\dagger)$, and $W^\ast=W(p,t_L^\ast, x_L^\ast, t_H^\ast, x_H^\ast)$ respectively.
Notice the slope of the linear pieces of $\hat{W}^\ast$ and $W^\ast$; in
particular, $\hat{W}^\ast(p)$ for $p\geq \hat{p}^\ast$ is increasing at a slower
rate than $W^\ast(p)$ for $p\geq p^\ast$. Similarly, $\Pi^\ast(p)$ for $p\geq p^\ast$
   increases at a faster rate than $\hat{\Pi}^\ast(p)$ for $p\geq \hat{p}^\ast$. This is in part due to the fact that
 $t_H^\ast(p)-t_L^\ast(p)>\htt_H^\ast(p)-\htt_L^\ast(p)$ as is shown in
 Figure~\ref{fig:prices}. We remark that there are some
   values for $p$ for which the utility company's profit and the social welfare
   are lower when risk is considered; %there is
   %risk; 
   this motivates adding compensation or insurance as a function of the
   population distribution into the contract.

\section{Discussion}
\label{sec:discussion}
As the capability of smart meters to collect data at high frequencies increases, we need to develop tools so that consumers and utilities benefit from these advances. 
%privacy metrics and privacy-aware sampling policies are needed. 
Implementing privacy-aware data collection policies results in a reduction in the
fidelity of the data and hence, a reduction in the efficiency of grid operations that depend on that data.
This fundamental tradeoff provides an incentive for the utility company to offer
new service contracts.

In this paper we modeled electricity service as a product line differentiated according to
privacy. In this setting, consumers can self-select the level of privacy that
fits their needs and wallet. We derived privacy contracts both when loss risks
are considered and when they are not. We characterized the optimal
solution in each of the cases and provided a comparative study.
We showed that loss risks decrease the level service offered to each consumer type.

%We provided the optimal contract solution for a consumer base composed of only two types of consumer. It is known in the theory of contracts that the
%fundamental characteristics of the two-type model are the same for the case when there is a
%continuum of types (see, e.g.,~\cite{bolton:2005aa}). In particular, the following characteristics, which we showed the two-type privacy contract possesses, will also hold for a continuum of types: the high-type recieves an efficient
%allocation and positive information rent whereas the low-type gets zero
%surplus. %, and all types except the
%low-type receive positive information rents. 
%In the case where there is a continuum of types 

We remark that people who value high privacy more, need to be compensated more to
participate in the smart grid. If there are two contracts, then even consumers
who do not value privacy much will have an incentive to lie. Through the
screening mechanism, the consumer will report their type truthfully.
 The screening process is a way to do customer segmentation, the result of which can lead to targeting.
 %In this setting, privacy is viewed as a good and
 %consumers can select the quality of the service contract with the utility
 %company. 
 %differentiated according to 
 %hey require high-fidelity data for efficient
 %operations; yet, consumers face the loss of privacy when such fine-grained data
 %is being collected
%
%Then using these probabilities we can
%design a screening mechanism consisting of a menu of contracts to be offered to
%consumers. One set of contracts to be offered by the utility company assess how the consumer values privacy thereby revealing his preferences. 
%hidden \emph{type}. 
In particular, using knowledge of consumer preferences, the utility company
could then incentivize consumers based on their preferences to choose a low
privacy setting thereby increasing the granularity of data.

Further, we showed that the utility company has an incentive to purchase insurance and
invest in security when there are loss risks. We leave the questions around how much the utility company
should invest in insurance versus security for future work. There are a number of open questions in the design of insurance contracts to be offered to the utility company by a third party insurer given the
consumer faces loss risk. We have made initial efforts at studying insurance
contracts to be offered to the consumer~\cite{ratliff:2014aa}; however, there is
much to be done in understanding how the optimal contracts vary as a function of
the distribution of types and the privacy metric used. 

%In our framework, we assumed the utility company had a prior on the types in its consumer base. Developing an algorithm for learning this prior in a dynamic setting is an interesting direction for future research.

Other researchers have used contract theory for demand-side management such as DLC and demand response programs.  Given that smart meters can collect data at high frequencies, it would be interesting to consider the design of contracts with multiple goods --- e.g. privacy setting, DLC options --- in a multidimensional screening problem.  In such a setting, we may also model the consumer's private information (type) as multidimensional vector thereby increasing the practical relevance of the model. Further, Assumption~\ref{ass:marginal} is often referred to as the \emph{sorting condition}. One of the major difficulties in extending to the multidimensional case is the lack of being able to sort or compare across the different goods and their qualities~\cite{basov:2005aa}; however, a potential solution is to create a partial order of the multiple goods (benefits and privacy) available to consumers. 

In conclusion, there are multiple future research directions we can explore from the model we have presented in this paper. Our model provides a general a mathematical framework for considering privacy as part of a service contract between an electric utility and the consumer. 
%In this framework we study the effects of loss risk and population distribution on the contract solution. 
This line of research can help inform data collection schemes and privacy policy in the smart grid.
%which allows the utility to screen the consumer

%Designing contracts in which electricity 
%In order to bring our results into practice, the next step would be to consider 
%we leave
%the full characterization of the solution to the insurance design problem with
%risk for future work.
%\IEEEtriggeratref
%\section{Privacy Contracts with Compensation}
%\label{sec:compensation}
%\input{contractswithriskv2}
%\section{Conclusion}
%The conclusion goes here.

\balance

%\vfill\break
% or


%\appendices
%
%% use section* for acknowledgement
%\section*{Acknowledgment}

\ifCLASSOPTIONcaptionsoff
  \newpage
\fi
\bibliographystyle{IEEEtran}
\bibliography{2014IEEETSG}

\end{document}